\documentclass[12pt,dvipdfm]{article}
\usepackage{amsfonts}
\usepackage{graphicx}
\usepackage{amsfonts, amsmath, amssymb, latexsym, epsfig}
\usepackage{mathrsfs}
\newcommand{\blst}{\begin{trivlist}}
\newcommand{\elst}{\end{trivlist}}

\newtheorem{thm}{Theorem}[section]
\newtheorem{prop}[thm]{Proposition}
\newtheorem{cor}[thm]{Corollary}
\newtheorem{lem}[thm]{Lemma}
\newtheorem{conj}[thm]{Conjecture}
\newtheorem{exa}[thm]{Example}
\newtheorem{defn}[thm]{Definition}

\newcommand{\ben}{\begin{enumerate}}
\newcommand{\een}{\end{enumerate}}
\newcommand{\ble}{\begin{lem}}
\newcommand{\ele}{\end{lem}}
\newcommand{\bth}{\begin{thm}}
\renewcommand{\eth}{\end{thm}}
\newcommand{\bpr}{\begin{prop}}
\newcommand{\epr}{\end{prop}}
\newcommand{\bco}{\begin{cor}}
\newcommand{\eco}{\end{cor}}
\newcommand{\bcon}{\begin{conj}}
\newcommand{\econ}{\end{conj}}
\newcommand{\bde}{\begin{defn}}
\newcommand{\ede}{\end{defn}}
\newcommand{\bex}{\begin{exa}}
\newcommand{\eex}{\end{exa}}
\newcommand{\barr}{\begin{array}}
\newcommand{\earr}{\end{array}}
\newcommand{\btab}{\begin{tabular}}
\newcommand{\etab}{\end{tabular}}
\newcommand{\beq}{\begin{equation}}
\newcommand{\eeq}{\end{equation}}

\newcommand{\bea}{\begin{eqnarray*}}
\newcommand{\eea}{\end{eqnarray*}}
\newcommand{\beaa}{\begin{eqnarray}}
\newcommand{\eeaa}{\end{eqnarray}}
\newcommand{\bce}{\begin{center}}
\newcommand{\ece}{\end{center}}
\newcommand{\bpi}{\begin{picture}}
\newcommand{\epi}{\end{picture}}
\newcommand{\bfi}{\begin{figure} \begin{center}}
\newcommand{\efi}{\end{center} \end{figure}}

\newcommand{\bsl}{\begin{slide}{}}
\newcommand{\esl}{\end{slide}}

{}
{}
{}
{}
{}
{}
{}
{}
{}
{}
{}
{}
{}
{}
{}
{}
{}
{}
{}
{}
{}
{}
{}
{}
{}
{}
{}
{}
\newenvironment{proof}{
\par
\noindent {\bf Proof.}\rm}{\mbox{}\hfill\rule{0.5em}{0.809em}\par}
\begin{document}
\title{Generalizations of The Chung-Feller Theorem}

\author{
Jun Ma$^{a,}$\thanks{Email address of the corresponding author:
majun@math.sinica.edu.tw}
 \and Yeong-Nan Yeh $^{b,}$\thanks{Partially supported by NSC 96-2115-M-001-005}}

\date{}
\maketitle \vspace{-1cm} \bce \footnotesize
 $^{a,b}$ Institute of Mathematics, Academia Sinica, Taipei, Taiwan\\
\ece

\thispagestyle{empty}\vspace*{.4cm}

\begin{abstract} The classical Chung-Feller theorem [2] tells us that
the number of Dyck paths of length $n$ with flaws $m$ is the $n$-th
Catalan number and independent on $m$. L. Shapiro [7] found the
Chung-Feller properties for the Motzkin paths. In this paper,  we
find the connections between these two Chung-Feller theorems. We
focus on the weighted versions of three classes of lattice paths and
give the generalizations of the above two theorems. We prove the
Chung-Feller theorems of Dyck type for these three classes of
lattice paths and the Chung-Feller theorems of Motzkin type for two
of these three classes. From the obtained results, we find an
interesting fact that many lattice paths have the Chung-Feller
properties of both Dyck type and Motzkin type.
\end{abstract}


\section{Introduction}
Let $\mathcal{S}$ be a subset of the set $\mathbb{Z}\times
\mathbb{Z}\setminus\{(0,0)\}$, where $\mathbb{Z}$ is the set of the
integers. We call $\mathcal{S}$ the {\it step set}. Let $k$ be an
integer.
\begin{defn} An $(\mathcal{S},k)$-lattice path is a path in
$\mathbb{Z}\times \mathbb{Z}$ which:

(a) is made only of steps in $\mathcal{S}$;

(b) starts at $(0,0)$ and ends on the line $y=k$.\\
If it is made of $m$ steps and ends at $(r,k)$, we say that it is of
order $m$ and size $r$.
\end{defn}
Let $\mathscr{L}^k$ be the set of all the $(\mathcal{S},k)$-lattice
path. For short we call $(\mathcal{S},0)$-lattice paths
$\mathcal{S}$-paths  and write $\mathscr{L}^0$ as $\mathscr{L}$. Let
$w$ and $l$ be two mappings from $\mathcal{S}$ to $\mathbb{R}$,
where $\mathbb{R}$ is the set of the real number. We say that $w$
and $l$ are the weight function and the length function of
$\mathcal{S}$ respectively. For any $s\in\mathcal{S}$, $w(s)$ and
$l(s)$ are called the {\it weight} and the {\it length} of the step
$s$ respectively. We can view an $(\mathcal{S},k)$-lattice path $L$
of order $m$ as a word $s_{1}s_{2}\ldots s_{m}$, where
$s_{j}\in\mathcal{S}$. In this word, let $s_{j}$ denote the $j$-th
letter from the left. Define the weight $w(L)$ and the length $l(L)$
of the path $L$ as
$$w(L)=\prod\limits_{j=1}^mw(s_{j})\text{ and }l(L)=\sum\limits_{j=1}^ml(s_{j}).$$ Moreover, we can consider
the path $L$ as a sequence of the points
$$(0,0)=(x_0,x_0),(x_1,y_1),(x_2,y_2),\ldots,(x_m,y_m),$$ where
$(x_{j},y_{j})$ is the end point of the step $s_{j}$ in the lattice
path $L$ for $j\geq 1$. Let $h(s_{j})=h_L(s_{j})=y_{j}$ for all
$j\geq 1$. We say that $h(s_{j})$ is the {\it height} of the step
$s_{j}$ in $L$. Define
$$\bar{l}(L)=\sum\limits_{s_{j}\in L,h(s_{j})\leq 0}l(s_{j}).$$
$\bar{l}(L)$ is called the {\it non-positive length} of $L$. A {\it
minimum point} is a point $(x_i,y_i)$ in the path $L$ such that
$y_i\leq y_j$ for all $j\neq i$, let $m(L)=y_i$, we call $m(L)$ the
{\it minimum value} of $L$. A {\it absolute minimum point} is a
minimum point $(x_i,y_i)$ such that the point is the rightmost one
among all the minimum points, and the index $i$ is called the {\it
absolute minimum position}, denoted by $mp(L)$. Finally, we define
the {\it absolute minimum length} $ml(L)$ of the path $L$ as
$$ml(L)=\sum\limits_{1\leq j\leq mp(L)}l(s_j).$$
\begin{defn} An $(\mathcal{S},k)$-nonnegative path is an $(\mathcal{S},k)$-lattice path
which never goes below the line $y=k$.
\end{defn}
Let $\mathscr{N}^k$ be the set of all the
$(\mathcal{S},k)$-nonnegative path. For short we call
$(\mathcal{S},0)$-nonnegative path $\mathcal{S}$-nonnegative path
and  write $\mathscr{N}^0$ as $\mathscr{N}$.

Now, we set $\mathcal{S}=\{(1,1),(1,-1)\}$, $w(s)=1$ for
$s\in\mathcal{S}$, $l((1,1))=1$ and $l((1,-1))=0$. In this situation
, an $(\mathcal{S},0)$-nonnegative path is called Dyck path as well.
We may state the classical Chung-Feller Theorem [2] as follows:

{\it The number of $(\mathcal{S},0)$-lattice paths with length $n$
and non-positive length $m$ is equal to the number of the Dyck paths
with length $n$ and independent on $m$.}

It is well known that the number of Dyck paths with length $n$ is
the $n$-th Catalan number $c_n=\frac{1}{n+1}{2n\choose{n}}$. The
generating function $C(z):=\sum_{n\ge 0}c_n z^n$ satisfies the
functional equation $C(z)=1+zC(z)^2$ and
$C(z)=\frac{1-\sqrt{1-4z}}{2z}$ explicitly.

The Chung-Feller Theorem were proved by using analytic method in
[2]. T.V.Narayana [6] showed the Chung-Feller Theorem by
combinatorial methods. S.P.Eu et al. [3] proved the Chung-Feller
Theorem by using the Taylor expansions of generating functions and
gave a refinement of this theorem. In [4], they gave a strengthening
of the Chung-Feller Theorem and a weighted version for Schr\"{o}der
paths. Y.M. Chen [1] revisited the Chung-Feller Theorem by
establishing a bijection.

Moreover, if we set $\mathcal{S}=\{(1,1),(1,-1),(1,0)\}$, $w(s)=1$
and $l(s)=1$ for $s\in\mathcal{S}$, then
$(\mathcal{S},0)$-nonnegative paths are the famous Motzkin paths. L.
Shapiro [7] found the following Chung-Feller phenomenons for the
Motzkin paths.

{\it The number of $(\mathcal{S},1)$-lattice paths with length $n+1$
and absolute minimum length $m$ is equal to the number of the
Motzkin paths with length $n$ and independent on m.}

It is well known that the number of Motzkin paths with length $n$ is
the $n$-th Motzkin number $m_n$. The generating function
$M(z):=\sum_{n\ge 0}m_n z^n$ satisfies $M(z)=1+zM(z)+z^2M(z)^2$ and
explicitly $M(z)=\frac{1-z-\sqrt{1-2z-3z^2}}{2z^2}$. Recently,
Shu-Chung Liu et al. [5] use an unify algebra approach to prove
chung-Feller theorems for Dyck path and Motzkin path and develop a
new method to find some combinatorial structures which have  the
Chung-Feller property.

The direct motivations of this paper come from the following two
problems:

(1) When $\mathcal{S}=\{(1,1),(1,-1)\}$, $w(s)=1$ for
$s\in\mathcal{S}$, $l((1,1))=1$ and $l((1,-1))=0$, is the number of
$\mathcal{S}$-paths with length $n$ and absolute minimum length $m$
independent on $m$ ?

(2) When $\mathcal{S}=\{(1,1),(1,-1),(1,0)\}$, $w(s)=1$ and $l(s)=1$
for $s\in\mathcal{S}$, is the number of $\mathcal{S}$-paths with
length $n$ and non-positive length $m$ independent on $m$?

We find that the answers of these two problems are yes. In fact, in
this paper, let $A$ and $B$ be two finite subsets of the set
$\mathbb{P}$, where $\mathbb{P}$ is the set of the positive
integers. We consider the weighted versions of the following three
classes of lattice paths.

{\bf Class 1.}
$\mathcal{S}_1=\mathcal{S}_{A}\cup\mathcal{S}_{B}\cup\{(1,1)\}$,
where $\mathcal{S}_{A}=\{(2i-1,-1)\mid i\in A\}$ and
$\mathcal{S}_{B}=\{(2i,0)\mid i\in B\}.$

 For any step $s\in
\mathcal{S}_1$, let
\begin{center}$l(s)=\left\{\begin{array}{lll}
i&\text{if}&s=(2i,0),\\
i-1&\text{if}&s=(2i-1,-1),\\
1&\text{if}&s=(1,1),
\end{array}\right.
$$w(s)=\left\{\begin{array}{lll}
b_i&\text{if}&s=(2i,0),\\
a_i&\text{if}&s=(2i-1,-1),\\
1&\text{if}&s=(1,1).\\
\end{array}\right.
$\end{center}

{\bf Class 2.}
$\mathcal{S}_2=\mathcal{S}_{A}\cup\mathcal{S}_{B}\cup\{(1,1)\}$,
where $\mathcal{S}_{A}=\{(i,-1)\mid i\in A\}$ and
$\mathcal{S}_{B}=\{(i,0)\mid i\in B\}.$

 For any step $s\in
\mathcal{S}_2$, let
\begin{center} $l(s)=\left\{\begin{array}{lll}
i&\text{if}&s=(i,0)\text{ and }(i,-1),\\
1&\text{if}&s=(1,1),\\
\end{array}\right.
$ $w(s)=\left\{\begin{array}{lll}
b_i&\text{if}&s=(i,0),\\
a_i&\text{if}&s=(i,-1),\\
1&\text{if}&s=(1,1).\\
\end{array}\right.
$\end{center}

 {\bf Class 3.}
$\mathcal{S}_3=\mathcal{S}_{A}\cup\mathcal{S}_{B}\cup\{(1,1)\}$,
where $\mathcal{S}_{A}=\{(1,-2i+1)\mid i\in A\}$ and
$\mathcal{S}_{B}=\{(2i,0)\mid i\in B\}.$

For any step $s\in \mathcal{S}_3$, let
\begin{center}
$l(s)=\left\{\begin{array}{lll}
i&\text{if}&s=(2i,0),\\
0&\text{if}&s=(1,-2i+1),\\
1&\text{if}&s=(1,1),\\
\end{array}\right.
$$w(s)=\left\{\begin{array}{lll}
b_i&\text{if}&s=(2i,0),\\
a_i&\text{if}&s=(1,-2i+1),\\
1&\text{if}&s=(1,1).\\
\end{array}\right.
$\end{center}

 First, we give the definition of the pointed lattice paths.
Then we define two parameters on the pointed lattice paths:
non-positive pointed length and absolute minimum pointed length. So,
for any step set $\mathcal{S}$, we say that the pointed
$(\mathcal{S},k)$-lattice paths have the {\it Chung-Feller
properties of Dyck type ( resp. Motzkin type)} if the sum of the
weights of all the pointed $(\mathcal{S},k)$-lattice paths with
length $n$ and non-positive pointed length ( resp. absolute minimum
pointed length) $m$ are independent on $m$. Finally, we prove the
Chung-Feller theorem of Dyck type for the above three classes of
lattice paths and the Chung-Feller theorem of Motzkin type for
Classes 1 and 2. From the obtained results, we find an interesting
fact that many lattice paths have the Chung-Feller properties of
both Dyck type and Motzkin type. These results tell us that there
are closed relations between two parameters of lattice paths:
non-positive pointed length and absolute minimum pointed length.

This paper is organized as follows. In Section 2, we give the
definition of the pointed lattice path and the definitions of two
parameters on the pointed lattice path: non-positive pointed length
and absolute minimum pointed length. In Section 3, we prove the
Chung-Feller Theorem of Dyck type for Classes 1,2,3. In Section 4,
we prove the Chung-Feller Theorem of Motzkin type for Classes 1,2.
In Section 5, we give some interesting facts and problems.

\section{The pointed path}
Throughout the paper, we let the step set $\mathcal{S}_i$ as well as
the corresponding weight function $w$ and
 length function $l$ be defined as that in Classes $i=1,2,3$.
In this section, we will give the definition of the pointed lattice
path and the definitions of two parameters of the pointed lattice
path: non-positive pointed length and absolute minimum pointed
length.

Let $L=s_{1}s_{2}\ldots s_{m}$ be an $(\mathcal{S}_i,k)$-lattice
path with $l(s_{m})\geq 1$. Recall that $L$ can be viewed as a
sequence of the points
$$(0,0)=(x_{0},x_{0}),(x_{1},y_{1}),(x_{2},y_{2}),\ldots,(x_{m},y_{m}),$$ where
$(x_{j},y_{j})$ is the end point of the step $s_{j}$ in the lattice
path $L$ for $j\geq 1$. For Classes $1$ and $3$ let
$\dot{L}=[L,(x_{m}-2j,0)]$ and for Classes $2$ let
$\dot{L}=[L,(x_{m}-j,0)]$ for some $0\leq j\leq l(s_{m})-1$,
moreover, let $p(\dot{L_i})=j$. $\dot{L}$ is called the {\it pointed
path} since it denote the the path $L$  marked by a point on the
$x$-axis and $p(\dot{L})$ is called the {\it pointed length} of
$\dot{L}_i$. Let $\mathscr{M}_i^k$ denote the set of all the pointed
$(\mathcal{S}_i,k)$-lattice path  in which the length of the final
step is no less than $1$. For short we write $\mathscr{M}_i^0$ as
$\mathscr{M}_i$.
\begin{defn} Given a path
$\dot{L}\in{\mathscr{M}_i^1}$, let
$\overline{lp}(\dot{L})=\bar{l}(L)+p(\dot{L})$.
$\overline{lp}(\dot{L})$ is called the {\it non-positive pointed
length} of $\dot{L}$. Let $mlp(\dot{L})=ml(L)+p(\dot{L})$.
$mlp(\dot{L})$ is called the {\it absolute minimum  pointed length}
of $\dot{L}$.
\end{defn}
\begin{exa} Let $\mathcal{S}=\{(1,1),(1,0),(5,-1)\}$, $l((1,1))=1$, $l((1,0))=1$ and $l((5,-1))=5$.
We draw a pointed $\mathcal{S}$-path $\dot{L}$ of length $14$ as
follows, where $*$ denote the marked point.
\begin{center}
\includegraphics[width=10cm]{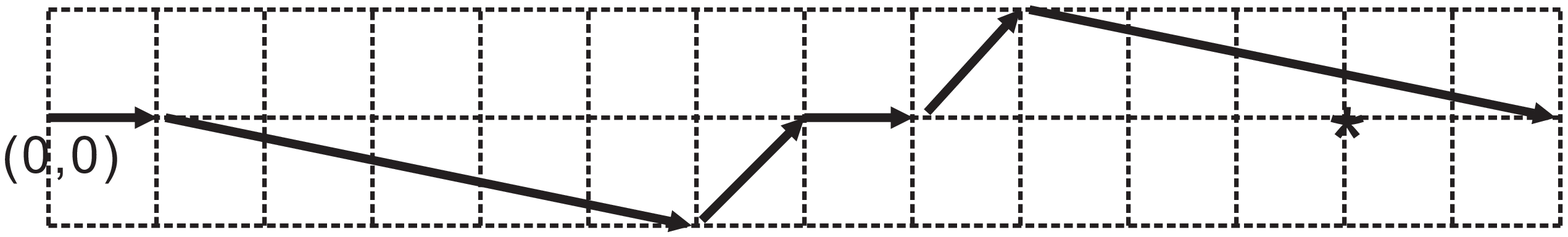}\\
Fig.1. A pointed $\mathcal{S}$-path $\dot{L}$ of length $14$
\end{center}
Note that the path is in Class $2$. So, it is easy to see $l(L)=14$
$p(\dot{L})=2$, $\bar{l}(L)=8$, and $ml(L)=6$. Hence,
$\overline{lp}(\dot{L})=10$ and $\overline{mlp}(\dot{L})=8$.
\end{exa}
For Classes $i=1,2,3$, define the generating functions
\begin{eqnarray*}D_i(y,z)&=&\sum\limits_{\dot{L}\in{\mathscr{M}_i^1}}w(L)y^{\overline{lp}(\dot{L})}z^{l(L)-1}\\
M_i(y,z)&=&\sum\limits_{\dot{L}\in{\mathscr{M}_i^1}}w(L)y^{mlp(\dot{L})}z^{l(L)-1}.
\end{eqnarray*}
Let $\bar{f}_{i;n,m}$ ( resp. $\bar{g}_{i;n,m}$) be the sum of the
weights of the pointed $(\mathcal{S}_i,1)$-lattice paths
$\dot{L}\in\mathscr{M}_i^1$ with length $n+1$ and non-positive
pointed length ( resp. absolute minimum pointed length ) $m$ for
$(n,m)\neq (0,0)$ and $\bar{f}_{i;0,0}=\bar{g}_{i;0,0}=1$. It is
easy to see
\begin{eqnarray}D_i(y,z)=\sum\limits_{n\geq 0}\sum\limits_{m=
0}^n\bar{f}_{i;n,m}y^mz^n.\end{eqnarray} and
\begin{eqnarray}M_i(y,z)=\sum\limits_{n\geq 0}\sum\limits_{m=
0}^n\bar{g}_{i;n,m}y^mz^n.\end{eqnarray}

Let $\mathscr{N}_i$ be the set of all the
$\mathcal{S}_i$-nonnegative path. Define the generating function
$$F_i(z)=\sum\limits_{L\in\mathscr{N}_i}w(L)z^{l(L)}.$$
\begin{lem}\label{Dycktypegenerating} For Classes 1,2 and 3, we have
\\
(1) $F_1(z)=1+\left(\sum\limits_{i\in
B}b_{i}z^i\right)F_1(z)+\left(\sum\limits_{i\in A}a_{i}z^i\right)[F_1(z)]^2$\\
(2) $F_2(z)=1+\left(\sum\limits_{i\in
B}b_{i}z^i\right)F_2(z)+\left(\sum\limits_{i\in
A}a_{i}z^{i+1}\right)[F_2(z)]^2,$\\
(3) $F_3(z)=1+\left(\sum\limits_{i\in
B}b_{i}z^i\right)F_3(z)+\left(\sum\limits_{i\in
A}a_{i}z^i\right)[F_3(z)]^{i+1}. $
\end{lem}
\begin{proof} (1) Given a path $L\in \mathscr{N}_1$ and $L\neq
\emptyset$, we suppose that $s$ is the first step of ${L}$ and
discuss the following two cases:

{\it Case I.} $s=(2i,0)$ for some $i\in B$. We can decompose the
path $L$ into $sR$, where $R\in\mathscr{N}_1$. Note that $l(s)=i$
and $w(s)=b_i$. This provide the term $\left(\sum\limits_{i\in
B}b_iz^i\right)F_1(z)$.

{\it Case II.} $s=(1,1)$. Let $t$ be the first step returning to the
$x$-axis. We can decompose the path $L$ into $sRtQ$, where
$R,Q\in\mathscr{N}_1$ and $t=(2i-1,-1)$ for some $i\in A$. Note that
$l(s)=1$ and $w(s)=1$, $l(t)=i-1$ and $w(t)=a_i$. This provide the
term $\left(\sum\limits_{i\in A}a_iz^i\right)[F_1(z)]^2$.

(2) The proof is similar to that of (1).

(3) Given a path $L\in \mathscr{N}_3$ and $L\neq \emptyset$, we
suppose that $s$ is the first step of ${L}$ and discuss the
following two cases:

{\it Case I.} $s=(2i,0)$ for some $i\in B$. Similar to Case I in
(1), we can obtain the term $\left(\sum\limits_{i\in
B}b_iz^i\right)F_3(z)$.

{\it Case II.} $s\notin\mathcal{S}_{B}$. Let $t=(1,-2i+1)$ be the
first step returning to the $x$-axis for some $i\in A$. Let $s_j$ be
the first step $(1,1)$ of height $j$ for $1\leq j\leq i$. We can
decompose the path $L$ into $s_1R_1s_2R_2\ldots s_iR_itQ$, where
 $R_j\in\mathscr{N}_3$ for all $j$ and $Q\in\mathscr{N}_3$. Note that
$l(s_j)=1$ and $w(s_j)=1$ for all $j$, $l(t)=0$ and $w(t)=a_i$. This
provide the term $\left(\sum\limits_{i\in
A}a_iz^i\right)[F_3(z)]^{i+1}$.
\end{proof}

For Classes $i=1,2,3$, let ${f}_{i;n}$ be the sum of the weights of
the $\mathcal{S}_i$-nonnegative paths with length $n$ for $n\geq 1$
and ${f}_{i;0}=1$. It is easy to see
\begin{eqnarray}F_i(z)=\sum\limits_{n\geq 0}{f}_{i;n}z^n.\end{eqnarray}

\section{The Chung-Feller property of Dyck type}
In this section, we will prove the Chung-Feller Theorem of Dyck type
for Classes $1,2,3$. For $i=1,2,3$, let $\mathscr{P}_i$ be the set
of all the pointed $\mathcal{S}_i$-nonnegative path in which the
length of the final step is no less than $1$ and define the
generating function
$$P_i(y,z)=\sum\limits_{\dot{L}\in
\mathscr{P}_i}w(L)y^{p(\dot{L})}z^{l(L)}.$$

\begin{lem}\label{dycktypegeneragtingpoint} For Classes $1,2,3,$ we
have\\
 (1) $P_1(y,z)=1+\left(\sum\limits_{i\in
B}b_iz^i\sum\limits_{j=0}^{i-1}y^{j}\right)F_1(z)+\left(\sum\limits_{i\in
A}a_iz^i\sum\limits_{j=0}^{i-2}y^{j}\right)[F_1(z)]^2$\\
(2) $ P_2(y,z)=1+\left(\sum\limits_{i\in
B}b_iz^i\sum\limits_{j=0}^{i-1}y^{j}\right)F_2(z)+\left(\sum\limits_{i\in
A}a_iz^{i+1}\sum\limits_{j=0}^{i-1}y^{j}\right)[F_2(z)]^2$\\
(3) $ P_3(y,z)=1+\left(\sum\limits_{i\in
B}b_iz^i\sum\limits_{j=0}^{i-1}y^{j}\right)F_3(z).$
\end{lem}

\begin{proof} (1) Given a path $\dot{L}\in \mathscr{P}_1$ and $\dot{L}\neq
\emptyset$, we suppose that $s$ is the final step of $\dot{L}$ and
$(x,0)$ is the final point which the path $\dot{L}$ reach. Then
$\dot{L}=[L,(x-2j,0)]$ for some $0\leq j\leq l(s)-1$. We discuss the
following two cases:

{\it Case I.} $s=(2i,0)$ for some $i\in B$. We can decompose the
path $L$ into $Rs$, where $R\in\mathscr{N}_1$. Note that $l(s)=i$ ,
$w(s)=b_i$ and $j\in\{0,1,\ldots,i-1\}$. This provide the term
$\left(\sum\limits_{i\in
B}b_iz^i\sum\limits_{j=0}^{i-1}y^{j}\right)F_1(z)$.

{\it Case II.} $s=(2i-1,-1)$ for some $i\in A$ and $i\geq 2$. Let
$t$ be the right-most step leaving the $x$-axis. We can decompose
the path $L$ into $QtRs$, where $R,Q\in\mathscr{N}_1$ and $t=(1,1)$.
Note that $l(t)=1$, $w(t)=1$, $l(s)=i-1$, $w(s)=a_i$ and
$j\in\{0,1,\ldots,i-2\}$. This provide the term
$\left(\sum\limits_{i\in
A}a_iz^i\sum\limits_{j=0}^{i-2}y^{j}\right)[F_1(z)]^2$.

(2) The proof is similar to that of (1).

(3) For any $\dot{L}\in\mathscr{P}_3$, suppose $s$ is the final step
of $\dot{L}$. Clearly, $s\neq (1,1)$. Furthermore, we have
$s=(2i,0)$ for some $i\in B$ since $l(s)\geq 1$. Using the similar
method as Case I in (1), we can obtain the identity as desired.

\end{proof}

Now, we turn to $\mathcal{S}_i$-path for Classes $i=1,2,3$. Let
$\mathscr{L}_i$ be the set of all the $\mathcal{S}_i$-paths. Define
the generating functions
$$G_i(y,z)=\sum\limits_{L\in\mathscr{L}_i}w(L)y^{\bar{l}(L)}z^{l(L)}.$$

\begin{lem}\label{dycktypegeneratingnofinal}
For Classes $1,2,3$, we have
\begin{eqnarray*}(1)~G_1(y,z)&=&1+\left(\sum\limits_{i\in
B}b_iy^iz^i\right)G_1(y,z)+\left(\sum\limits_{i\in
A}a_iy^iz^i\right)F_1(yz)G_1(y,z)\\
&&+\left(\sum\limits_{i\in
A}a_iy^{i-1}z^i\right)F_1(z)G_1(y,z),\\
Equivalently,~G_1(y,z)&=&\frac{1}{1-\sum\limits_{i\in
B}b_iy^iz^i-\left(\sum\limits_{i\in
A}a_iy^iz^i\right)F_1(yz)-\left(\sum\limits_{i\in
A}a_iy^{i-1}z^i\right)F_1(z)}\\
(2)~G_2(y,z)&=&1+\left(\sum\limits_{i\in
B}b_iy^iz^i\right)G_2(y,z)+\left(\sum\limits_{i\in
A}a_iy^{i+1}z^{i+1}\right)F_2(yz)G_2(y,z)\\
&&+\left(\sum\limits_{i\in
A}a_iy^{i}z^{i+1}\right)F_2(z)G_2(y,z),\\
Equivalently,~G_2(y,z)&=&\frac{1}{1-\sum\limits_{i\in
B}b_iy^iz^i-\left(\sum\limits_{i\in
A}a_iy^{i+1}z^{i+1}\right)F_2(yz)-\left(\sum\limits_{i\in
A}a_iy^{i}z^{i+1}\right)F_2(z)}\\
(3)~G_3(y,z)&=&1+\left(\sum\limits_{i\in
A}a_iz^{i}\sum\limits_{j=0}^{i}y^{i-j}[F_3(yz)]^{i-j}[F_3(z)]^j\right)G_3(y,z)\\
&&+\left(\sum\limits_{i\in B}b_iy^iz^i\right)G_3(y,z).\\
Equivalently,~ G_3(y,z)&=&\frac{1}{1-\sum\limits_{i\in
B}b_iy^iz^i-\sum\limits_{i\in
A}a_iz^{i}\sum\limits_{j=0}^{i}y^{i-j}[F_3(yz)]^{i-j}[F_3(z)]^j}
\end{eqnarray*}
\end{lem}
\begin{proof} (1) Given a path ${L}\in \mathscr{L}_1$ and ${L}\neq
\emptyset$, we suppose that $s$ is the first step of ${L}$. We
discuss the following three cases:

{\it Case I.} $s=(2i,0)$ for some $i\in B$. We can decompose the
path $L$ into $sR$, where $R\in\mathscr{L}_1$. Note that $l(s)=i$ ,
$w(s)=b_i$ and $h(s)=0$. This provide the term
$\left(\sum\limits_{i\in B}b_iy^iz^i\right)G_1(y,z)$.

{\it Case II.} $s=(2i-1,-1)$ for some $i\in A$. Let $t=(1,1)$ be the
left-most step returning to the $x$-axis. We can decompose the path
$L$ into $s\overline{Q}tR$, where $R\in\mathscr{L}_1$, and if we
view $\overline{Q}$ as a word $s_1s_2\ldots s_r$ of the steps in
$\mathcal{S}_1$, then $Q=s_{r}s_{r-1}\ldots s_1\in\mathscr{N}_1$.
Note that $l(t)=1$, $w(t)=1$, $l(s)=i-1$, $w(s)=a_i$, $h(s)=-1$ and
$h(t)=0$. This provide the term $\left(\sum\limits_{i\in
A}a_iy^iz^i\right)F_1(yz)G_1(y,z)$.

{\it Case III.} $s=(1,1)$. Let $t=(2i-1,-1)$ be the first step
returning to the $x$-axis. We can decompose the path $L$ into
$sQtR$, where $R\in\mathscr{L}_1$ and $Q\in\mathscr{N}_1$. Note that
$l(t)=i-1$, $w(t)=a_i$, $l(s)=1$, $w(s)=1$, $h(s)=1$ and $h(t)=0$.
This provide the term $\left(\sum\limits_{i\in
A}a_iy^{i-1}z^i\right)F_1(z)G_1(y,z)$.

(2) The proof is similar to that of (1).

(3) Given a path $L\in \mathscr{L}_3$ and $L\neq \emptyset$, we
suppose that $s$ is the first step of ${L}$. We discuss the
following two cases:

{\it Case I.} $s=(2i,0)$ for some $i\in B$. Similar to Case I in
(1), we can obtain the term $\left(\sum\limits_{i\in
B}b_iy^iz^i\right)G_3(y,z)$.

{\it Case II.} $s\notin\mathcal{S}_{B}$. Let $t=(1,-2i+1)$ be the
left-most step which passes the $x$-axis and has height $-i+j$ for
some $i\in A$ and $0\leq j\leq i$. Furthermore, let $s_r$ be the
first step $(1,1)$ with height $r$ for any $1\leq r\leq j$. Let
$u_r$ be the first step $(1,1)$ at the right of $t$ with height
$-r+1$ for any $1\leq r\leq i-j$. So, we can decompose the path $L$
into $s_1R_1s_2R_2\ldots
s_jR_jt\overline{Q}_1u_{i-j}\overline{Q}_2u_{i-j-1}\ldots
\overline{Q}_{i-j}u_1T$, where $R_r\in\mathscr{N}_3$ for all $1\leq
r\leq j$, $T\in\mathscr{L}_3$, and if we view $\overline{Q}_r$ as a
word $s'_{r,1}s'_{r,2}\ldots s'_{r,k}$ of the step in
$\mathcal{S}_3$, then $Q_r=s'_{r,k}s'_{r,k-1}\ldots
s'_{r,1}\in\mathscr{N}_3$ for all $1\leq r\leq i-j$. Note that
$l(s_r)=1$, $w(s_r)=1$, $h(s_r)\geq 1$ for all $1\leq r\leq j$,
$l(t)=0$, $w(t)=a_i$, $h(t)\leq 0$, $l(u_r)=1$, $w(u_r)=1$ and
$h(u_r)\leq 0$ for all $1\leq r\leq i-j$. This provide the term
$\left(\sum\limits_{i\in
A}a_iz^{i}\sum\limits_{j=0}^{i}y^{i-j}[F_3(yz)]^{i-j}[F_3(z)]^j\right)G_3(y,z)$.
\end{proof}

Recall that $\mathscr{M}_i^1$ is the set of all the pointed
$(\mathcal{S}_i,1)$-path in which the length of the final step is no
less than $1$  and
$D_i(y,z)=\sum\limits_{\dot{L}\in{\mathscr{M}_i^1}}w(L)y^{\overline{lp}(\dot{L})}z^{l(L)-1}$for
$i=1,2,3$.

\begin{lem}\label{dycktypegeneratingtheorem} For Classes $i=1,2,3,$ we
have
\begin{eqnarray*}D_i(y,z)=G_i(y,z)P_i(y,z).
\end{eqnarray*}
\end{lem}
\begin{proof} For any $i=1,2,3$, let $\dot{L}\in \mathscr{M}_i^1$.
Let $s$ be the right-most step $(1,1)$ leaving $x$-axis and reaching
the line $y=1$. We can decompose the path $\dot{L}$ into
$Rs\dot{Q}$, where $R\in\mathscr{L}_i$ and
$\dot{Q}\in\mathscr{P}_i$. Hence, $D_i(y,z)=G_i(y,z)P_i(y,z).$
\end{proof}

Now, we are in a position to prove the Chung-Feller theorem of Dyck
type for Classes 1,2,3.
\begin{thm}\label{dycktypechungfeller} For Classes $i=1,2,3,$ let
 $\bar{f}_{i;n,m}$ be the sum of the weights of the pointed
$(\mathcal{S}_i,1)$-lattice paths which \\
(a.) have length $n+1$,\\
(b.) have non-positive pointed length $m$,\\
(c.) have the length of the final step no less than $1$.\\
Let ${f}_{i;n}$ be the sum of the weights of the
$\mathcal{S}_i$-nonnegative paths with length $n$. Then
$\bar{f}_{i;n,m}$ has the Chung-Feller property of Dyck type, i.e.,
$\bar{f}_{i;n,m}={f}_{i;n}$.
\end{thm}
\begin{proof} First, we consider Class 1. By Lemmas 3.1,
3.2 and 3.3, we have
\begin{eqnarray*}D_1(y,z)&=&G_1(y,z)P_1(y,z)\\
&=&\frac{1+\left(\sum\limits_{i\in
B}b_iz^i\sum\limits_{j=0}^{i-1}y^{j}\right)F_1(z)+\left(\sum\limits_{i\in
A}a_iz^i\sum\limits_{j=0}^{i-2}y^{j}\right)[F_1(z)]^2}{1-\sum\limits_{i\in
B}b_iy^iz^i-\left(\sum\limits_{i\in
A}a_iy^iz^i\right)F_1(yz)-\left(\sum\limits_{i\in
A}a_iy^{i-1}z^i\right)F_1(z)}\\
&=&\frac{\left[1+\left(\sum\limits_{i\in
B}b_iz^i\sum\limits_{j=0}^{i-1}y^{j}\right)F_1(z)+\left(\sum\limits_{i\in
A}a_iz^i\sum\limits_{j=0}^{i-2}y^{j}\right)[F_1(z)]^2\right]F_1(yz)}{\left[1-\sum\limits_{i\in
B}b_iy^iz^i-\left(\sum\limits_{i\in
A}a_iy^iz^i\right)F_1(yz)-\left(\sum\limits_{i\in
A}a_iy^{i-1}z^i\right)F_1(z)\right]F_1(yz)}\\
&=&\frac{\left[1+\left(\sum\limits_{i\in
B}b_iz^i\sum\limits_{j=0}^{i-1}y^{j}\right)F_1(z)+\left(\sum\limits_{i\in
A}a_iz^i\sum\limits_{j=0}^{i-2}y^{j}\right)[F_1(z)]^2\right]F_1(yz)}{1-\left(\sum\limits_{i\in
A}a_iy^{i-1}z^i\right)F_1(z)F_1(yz)}
\end{eqnarray*}
since Lemma 2.3 tells us $F_1(yz)-\left(\sum\limits_{i\in
B}b_iy^iz^i\right)F_1(yz)-\left(\sum\limits_{i\in
A}a_iy^iz^i\right)[F_1(yz)]^2=1$. Note that
\begin{eqnarray*}&&\left[1+\left(\sum\limits_{i\in
B}b_iz^i\sum\limits_{j=0}^{i-1}y^{j}\right)F_1(z)+\left(\sum\limits_{i\in
A}a_iz^i\sum\limits_{j=0}^{i-2}y^{j}\right)[F_1(z)]^2\right]F_1(yz)\\
&=&\left[1+\left(\sum\limits_{i\in
B}b_iz^i\frac{y^i-1}{y-1}\right)F_1(z)+\left(\sum\limits_{i\in
A}a_iz^i\frac{y^{i-1}-1}{y-1}\right)[F_1(z)]^2\right]F_1(yz)\\
&=&\frac{1}{y-1}\left[y-1+\left(\sum\limits_{i\in
B}b_iz^i(y^i-1)\right)F_1(z)+\left(\sum\limits_{i\in
A}a_iz^i(y^{i-1}-1)\right)[F_1(z)]^2\right]F_1(yz)\\
&=&\frac{1}{y-1}\left[yF_1(yz)+\left(\sum\limits_{i\in
B}b_iz^iy^i\right)F_1(z)F_1(yz)+\left(\sum\limits_{i\in
A}a_iz^iy^{i-1}\right)[F_1(z)]^2F_1(yz)-F_1(z)F_1(yz)\right]
\end{eqnarray*}
Furthermore, since $\left(\sum\limits_{i\in
B}b_iz^iy^i\right)F_1(yz)=F_1(yz)-\left(\sum\limits_{i\in
A}a_iy^iz^i\right)[F_1(yz)]^2-1$, we
get\begin{eqnarray*}&&\left[1+\left(\sum\limits_{i\in
B}b_iz^i\sum\limits_{j=0}^{i-1}y^{j}\right)F_1(z)+\left(\sum\limits_{i\in
A}a_iz^i\sum\limits_{j=0}^{i-2}y^{j}\right)[F_1(z)]^2\right]F_1(yz)\\
&=&\frac{1}{y-1}\left[yF_1(yz)-\left(\sum\limits_{i\in
A}a_iy^iz^i\right)[F_1(yz)]^2F_1(z)-F_1(z)+\left(\sum\limits_{i\in
A}a_iz^iy^{i-1}\right)[F_1(z)]^2F_1(yz)\right]\\
&=&\frac{\left[yF_1(yz)-F_1(z)\right]\left[1-\left(\sum\limits_{i\in
A}a_iy^{i-1}z^i\right)F_1(z)F_1(yz)\right]}{y-1}.
\end{eqnarray*}Hence, \begin{eqnarray*}D_1(y,z)&=&\frac{yF_1(yz)-F_1(z)}{y-1}\\
&=&\frac{y\sum\limits_{n\geq 0}f_{1;n}y^nz^n-\sum\limits_{n\geq
0}f_{1;n}z^n}{y-1}\\
&=&\sum\limits_{n\geq 0}f_{1;n}z^n\frac{y^{n+1}-1}{y-1}\\
&=&\sum\limits_{n\geq 0}f_{1;n}z^n\sum\limits_{m=0}^ny^m\\
&=&\sum\limits_{n\geq 0}\sum\limits_{m=0}^nf_{1;n}y^mz^n.
\end{eqnarray*}
This implies $\bar{f}_{1;n,m}=f_{1;n}$ for all $0\leq m\leq n$.
Similarly, we can prove the theorems for Classes $i=2,3$.\end{proof}
\begin{cor}(Chung-Feller.) Let $\mathcal{S}=\{(1,1),(1,-1)\}$, $w(s)=1$ for any $s\in\mathcal{S}$,
$l((1,1))=1$ and $l((1,-1))=0$. Then the number of the
$(\mathcal{S},0)$-lattice path with length $n$ and non-positive
length $m$ is the $n$-th Catalan number.
\end{cor}
\begin{proof} For any a pointed
$(\mathcal{S},1)$-lattice path, we suppose that $s$ is the final
step in this path. Then $s=(1,1)$ since $l((1,-1))=0<1$. If we
delete the final step of this path and erase the marked point, we
will obtain an $(\mathcal{S},0)$-lattice path with length $n$ and
non-positive length $m$. By Theorem 3.4, the number of the pointed
$(\mathcal{S},1)$-lattice path with length $n+1$ and non-positive
pointed length $m$ in which the length of the final step is no less
than $1$  is equal to  the number of the $\mathcal{S}$-nonnegative
paths with length $n$. By Lemma 2.3, we have $F_1(z)=1+z[F_1(z)]^2$
since $\mathcal{S}=\{(1,1),(1,-1)\}$, $w(s)=1$ for any
$s\in\mathcal{S}$, $l((1,1))=1$ and $l((1,-1))=0$. Hence, the number
of the $\mathcal{S}_1$-nonnegative paths with length $n$ is the
$n$-th Catalan number. This complete the proof.
\end{proof}

\begin{cor}\label{mexample} Let $\mathcal{S}=\{(1,1),(1,-1),(1,0)\}$, $w(s)=1$ and $l(s)=1$ for any $s\in\mathcal{S}$.
 Then the number of the $(\mathcal{S},1)$-lattice path
with length $n+1$ and non-positive length $m$ is the $n$-th Motzkin
number.
\end{cor}
\begin{proof} For any a pointed
$(\mathcal{S},1)$-lattice path, we suppose that the final point in
this path is $(x,1)$. Then the marked point must be $(x,0)$ since
$l(s)=1$ for any $s\in\mathcal{S}$. So, we can erase the marked
point. By Theorem 3.4, the number of the pointed
$(\mathcal{S},1)$-lattice path with length $n+1$ and non-positive
pointed length $m$ is equal to  the number of the
$\mathcal{S}$-nonnegative paths with length $n$. By Lemma 2.3, we
have $F_2(z)=1+zF_2(z)+z^2[F_2(z)]^2$ since
$\mathcal{S}=\{(1,1),(1,-1),(1,0)\}$, $w(s)=1$ and $l(s)=1$ for any
$s\in\mathcal{S}$. Hence, the number of the
$\mathcal{S}$-nonnegative paths with length $n$ is the $n$-th
Motzkin number. This complete the proof.
\end{proof}
\begin{center}
\includegraphics[width=10cm]{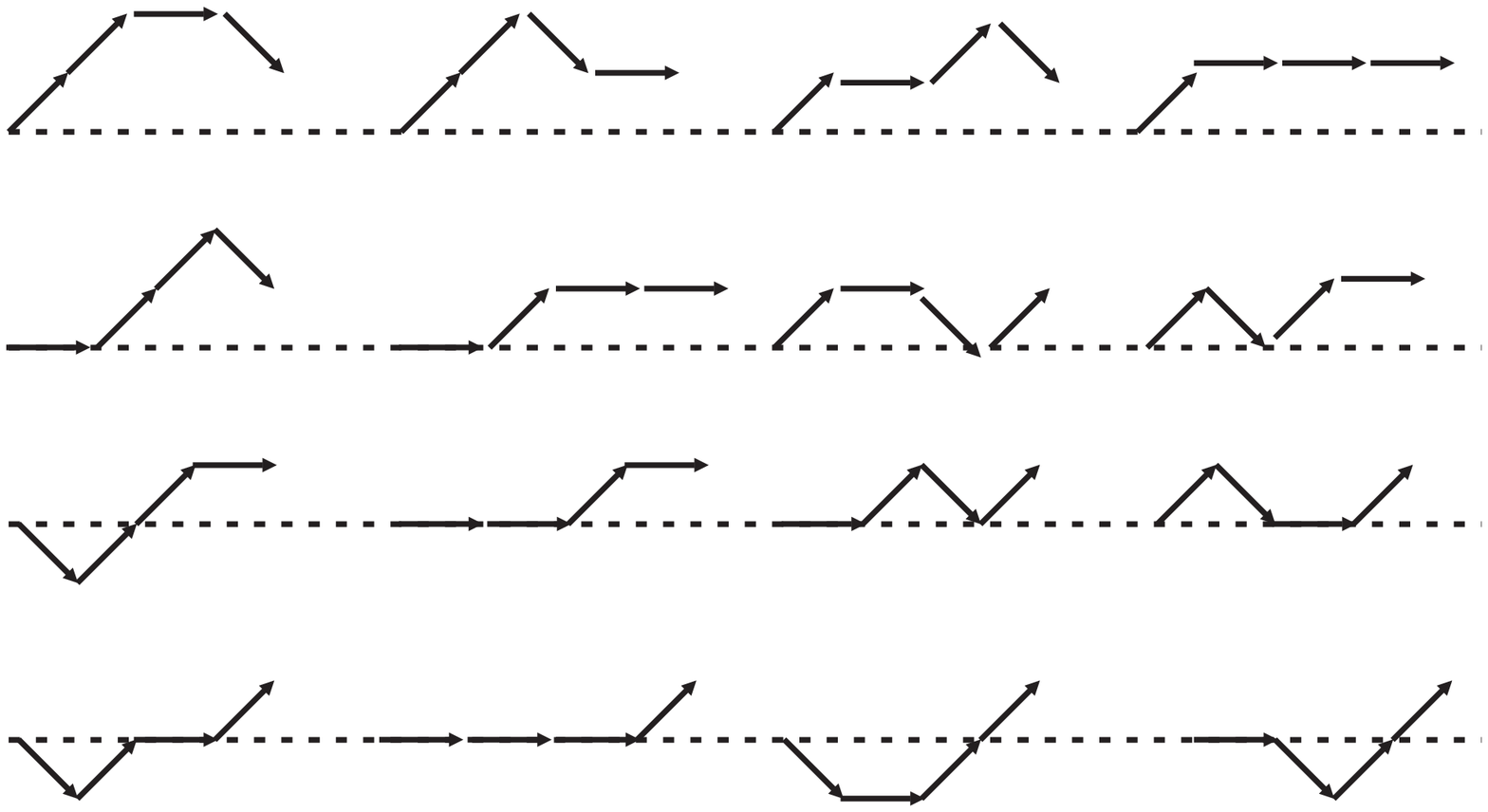}\\
Fig.2. An example of Corollary 3.6, where $n=4$
\end{center}
\begin{cor}\label{coro(5,-1)} Let $\mathcal{S}=\{(1,1),(5,-1),(1,-1)\}$, $w(s)=1$ for any $s\in\mathcal{S}$,
$l((1,1))=1$, $l((5,-1))=2$ and $l((1,-1))=0$. Then the number of
the pointed $(\mathcal{S},1)$-lattice path with length $n+1$ and
non-positive pointed length $m$ is equal to the number of the
$\mathcal{S}$-nonnegative path with length $n$.
\end{cor}
We omit the proof of Corollary 3.7. In Fig.3, we show an example of
this Corollary with $n=3$, where $*$ denote the marked point.
\begin{center}
\includegraphics[width=12cm]{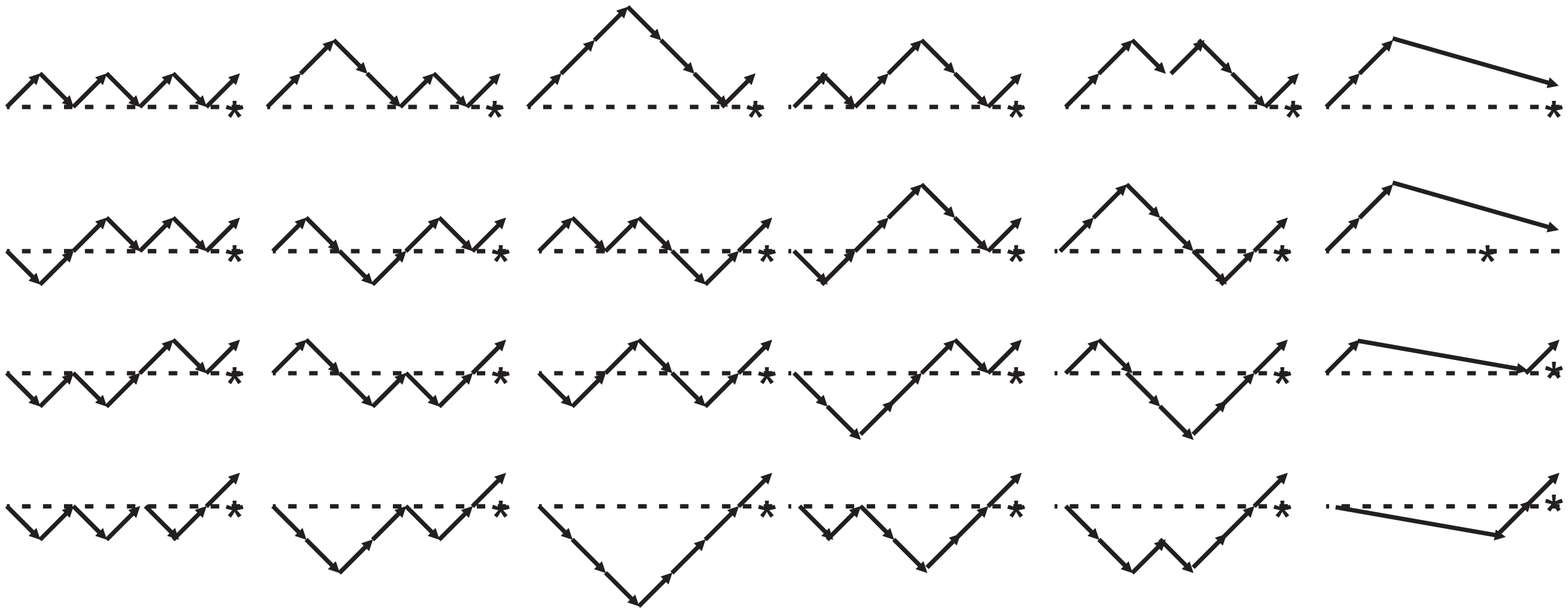}\\
Fig.3. An example of Corollary 3.7, where $n=3$.
\end{center}
\begin{cor}\label{coro(1,-3)} Let $\mathcal{S}=\{(1,1),(2,0),(1,-3)\}$, $w(s)=1$ for any $s\in\mathcal{S}$,
$l((1,1))=1$, $l((2,0))=1$ and $l((1,-3))=0$. Then the number of the
pointed $(\mathcal{S},1)$-lattice path with length $n+1$ and
non-positive length $m$ is equal to the number of the
$\mathcal{S}$-nonnegative path with length $n$.
\end{cor}
We omit the proof of Corollary 3.8. In Fig.4, we show an example of
this Corollary with $n=3$.
\begin{center}
\includegraphics[width=5cm]{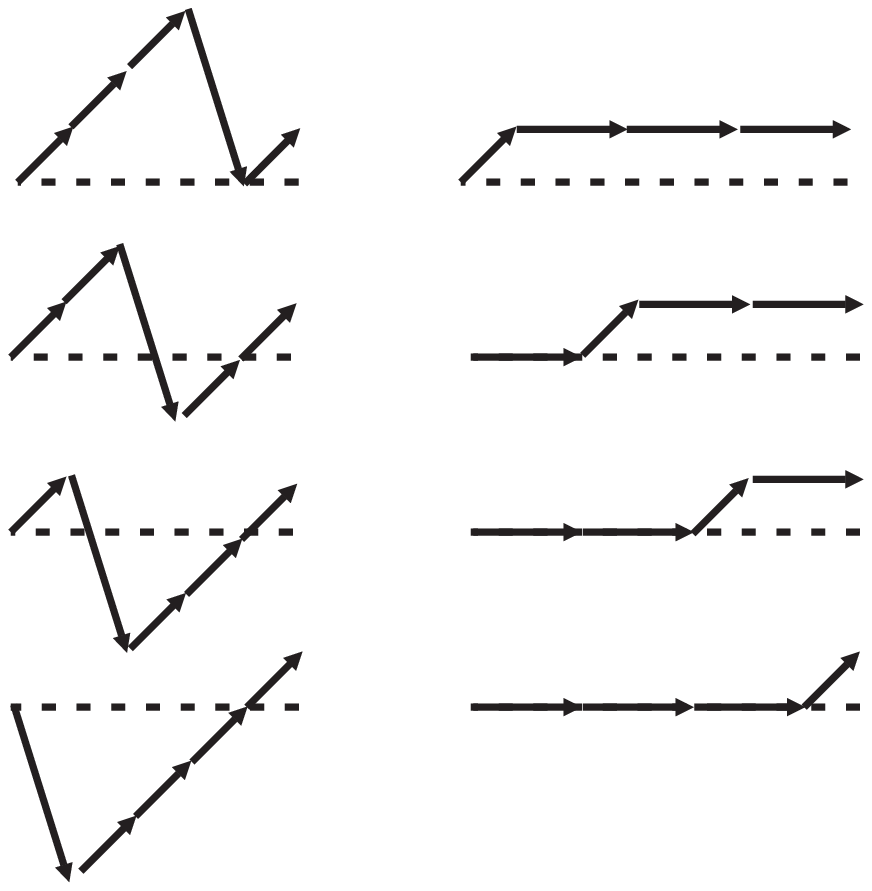}\\
Fig.4. An example of Corollary 3.8, where $n=3$.
\end{center}
\section{The Chung-Feller property of Motzkin
type} In this section, we will prove the Chung-Feller Theorem of
Motzkin type for Classes 1,2. For $i=1,2,$ recall that an
$(\mathcal{S}_i,-k)$-nonnegative path is an
$(\mathcal{S}_i,-k)$-lattice path which never goes below the line
$y=-k$, where $k\geq 0$. $\mathscr{N}_i^{-k}$ is the set of all the
$(\mathcal{S}_i,-k)$-nonnegative path. Define the generating
functions
$$H_{i}^k(z)=\sum\limits_{L\in \mathscr{N}_i^{-k}}w(L)z^{l(L)}.$$

\begin{lem}\label{motzkintypegeneratingminimum} For Classes 1,2, we
have
\begin{eqnarray*}H_1^k(z)=[F_1(z)]^{k+1}\left[\sum\limits_{i\in
A}a_iz^{i-1}\right]^k\text{ and
}H_2^k(z)=[F_2(z)]^{k+1}\left[\sum\limits_{i\in A}a_iz^{i}\right]^k.
\end{eqnarray*}
\end{lem}
\begin{proof} For any a path $L\in \mathscr{N}_i^{-k}$ and $L\neq\emptyset$, we
consider the first step $s_m$ with height $-m$, where $1\leq m\leq
k$. Thus we can decompose the path $L$ into
$L_0s_{1}{L}_1s_{2}\ldots {L}_{k-1}s_{k}{L}_{k}$, where
$L_r\in\mathscr{N}_i$ for all $0\leq r\leq k$ and
$s_{j}\in\mathcal{S}_{A}$ for all $j$. Thus,
\begin{eqnarray*}H_1^k(z)=[F_1(z)]^{k+1}\left[\sum\limits_{i\in
A}a_iz^{i-1}\right]^k\text{ and
}H_2^k(z)=[F_2(z)]^{k+1}\left[\sum\limits_{i\in A}a_iz^{i}\right]^k.
\end{eqnarray*}
\end{proof}

Now we focus on  the generating functions
$M_i(y,z)=\sum\limits_{\dot{L}\in{\mathscr{M}_i^1}}w(L)y^{mlp(\dot{L})}z^{l(L)-1}$
for $i=1,2$.

\begin{lem}\label{motzkintypegeneratingtheorem} For Classes 1,2, we
have
\begin{eqnarray*}
M_1(y,z)&=&\frac{P_1(y,z)F_1(yz)}{1-\left[\sum\limits_{i\in
A}a_iy^{i-1}z^{i}\right][F_1(yz)F_1(z)]},
\end{eqnarray*}and
\begin{eqnarray*}
M_2(y,z)&=&\frac{P_2(y,z)F_2(yz)}{1-\left[\sum\limits_{i\in
A}a_iy^{i}z^{i+1}\right][F_2(yz)F_2(z)]}.
\end{eqnarray*}
\end{lem}
\begin{proof} For any $k\geq 0$, let ${\mathscr{M}_i^1}(k)$ be the set of all the
paths $L$ in the set ${\mathscr{M}_i^1}$ such that $m(L)=-k$, where
$m(L)$ is the minimum value of $L$. Clearly,
${\mathscr{M}_i^1}=\bigcup\limits_{k\geq 0}{\mathscr{M}_i^1}(k)$.
Given a path $\dot{L}\in{\mathscr{M}_i^1}(k)$ and $L\neq \emptyset$,
using the absolute minimum position of $L$, we can decompose $L$
into $R(1,1)\dot{T}$, where $R\in\mathscr{N}_i^{-k}$. For the path
$\dot{T}$, we consider the rightmost step $(1,1)$ with height $-m$ ,
where $-1\leq m\leq k-1$. Thus we can decompose the path $T$ into
$L_{k-1}(1,1)L_{k-2}(1,1)\ldots L_{0}(1,1)\dot{Q}$, where $L_{j}$ is
$\mathcal{S}_i$-nonnegative path for all $0\leq j\leq k-1$ and
$\dot{Q}\in \mathscr{M}_i$, where $\mathscr{M}_i$ is the set of all
the pointed $\mathcal{S}_i$-nonnegative path. Hence, by Lemmas 3.1
and 4.1, we get \begin{eqnarray*} M_i(y,z)&=&\sum\limits_{k\geq
0}H_i^k(yz)[F_i(z)]^kz^kP_i(y,z).
\end{eqnarray*}Hence,
\begin{eqnarray*}
M_1(y,z)&=&P_1(y,z)F_1(yz)\sum\limits_{k\geq
0}[F_1(yz)]^{k}\left[\sum\limits_{i\in
A}a_iy^{i-1}z^{i-1}\right]^k[F_1(z)]^kz^k\\
&=&\frac{P_1(y,z)F_1(yz)}{1-\left[\sum\limits_{i\in
A}a_iy^{i-1}z^{i}\right][F_1(yz)F_1(z)]},
\end{eqnarray*}and
\begin{eqnarray*}
M_2(y,z)&=&P_2(y,z)F_2(yz)\sum\limits_{k\geq
0}[F_2(yz)]^{k}\left[\sum\limits_{i\in
A}a_iy^{i}z^{i}\right]^k[F_2(z)]^kz^k\\
&=&\frac{P_2(y,z)F_2(yz)}{1-\left[\sum\limits_{i\in
A}a_iy^{i}z^{i+1}\right][F_2(yz)F_2(z)]}.
\end{eqnarray*}
\end{proof}

Now, we can prove the following Chung-Feller theorem of Motzkin type
for Classes 1,2.

\begin{thm}\label{motzkintypechungfeller}
For Classes $i=1,2,$ let $\bar{g}_{i;n,m}$ be the sum of the weights
of the pointed
$(\mathcal{S}_i,1)$-lattice paths which \\
(a.) have length $n+1$,\\
(b.) have absolute minimum  pointed length $m$,\\
(c.) have the length of the final step no less than $1$.\\
Let ${f}_{i;n}$ be the sum of the weights of the
$\mathcal{S}_i$-nonnegative paths with length $n$. Then
$\bar{g}_{i;n,m}$ has the Chung-Feller property of Motzkin type,
i.e., $\bar{g}_{i;n,m}={f}_{i;n}$.
\end{thm}
\begin{proof} In fact, we derived
\begin{eqnarray*}M_i(y,z)&=&\frac{yF_i(yz)-F_i(z)}{y-1}.
\end{eqnarray*}in the
proof of Theorem 3.4 for $i=1,2$. Hence, the theorems
hold.\end{proof}
\begin{cor}(L. Shapiro) Let
$\mathcal{S}=\{(1,1),(1,-1),(1,0)\}$, $w(s)=1$ and $l(s)=1$ for any
$s\in\mathcal{S}$. Then the number of the $(\mathcal{S},1)$-lattice
path with length $n+1$ and absolute minimum  pointed length $m$  is
the $n$-th Motzkin number.
\end{cor}
\begin{proof} For any a pointed
$(\mathcal{S},1)$-lattice path, we suppose the final point in this
path is $(x,1)$. Then the marked point must be $(x,0)$ since
$l(s)=1$ for any $s\in\mathcal{S}$. So, we can erase the marked
point. By Theorem 4.3,  the number of the pointed
$(\mathcal{S},1)$-lattice path with length $n+1$ and absolute
minimum  pointed length $m$ is equal to  the number of the
$\mathcal{S}$-nonnegative paths with length $n$. By Lemma 2.3, we
have $F_2(z)=1+zF_2(z)+z^2[F_2(z)]^2$ since
$\mathcal{S}=\{(1,1),(1,-1),(1,0)\}$, $w(s)=1$ and $l(s)=1$ for any
$s\in\mathcal{S}$. Hence, the number of the
$\mathcal{S}$-nonnegative paths with length $n$ is the $n$-th
Motzkin number. This complete the proof.
\end{proof}

\begin{cor}\label{dyckexample} Let $\mathcal{S}=\{(1,1),(1,-1)\}$, $w(s)=1$ for any
$s\in\mathcal{S}$, $l((1,1))=1$ and $l((1,-1))=0$. Then the number
of the $\mathcal{S}$-path with length $n$ and absolute minimum
pointed length $m$ in which the length of the final step is no less
than $1$  is the $n$-th Catalan number.
\end{cor}
\begin{proof} For any a pointed
$(\mathcal{S},1)$-lattice path, we suppose the final step in this
path is $s$. Then $s=(1,1)$ since $l((1,-1))=0<1$. If we delete the
final step of this path and erase the marked point, we will obtain a
$\mathcal{S}$-path with length $n$ and absolute minimum length $m$.
By Theorem 4.3,  the number of the pointed $(\mathcal{S},1)$-lattice
path with length $n+1$ and absolute minimum  pointed length $m$ in
which the length of the final step is no less than $1$   is equal to
the number of the $\mathcal{S}$-nonnegative paths with length $n$.
By Lemma 2.3, we have $F_1(z)=1+z[F_1(z)]^2$ since
$\mathcal{S}=\{(1,1),(1,-1)\}$, $w(s)=1$ for any $s\in\mathcal{S}$,
$l((1,1))=1$ and $l((1,-1))=0$. Hence, the number of the
$\mathcal{S}$-nonnegative paths with length $n$ is the $n$-th
Catalan number. This complete the proof.
\end{proof}

\begin{center}
\includegraphics[width=6cm]{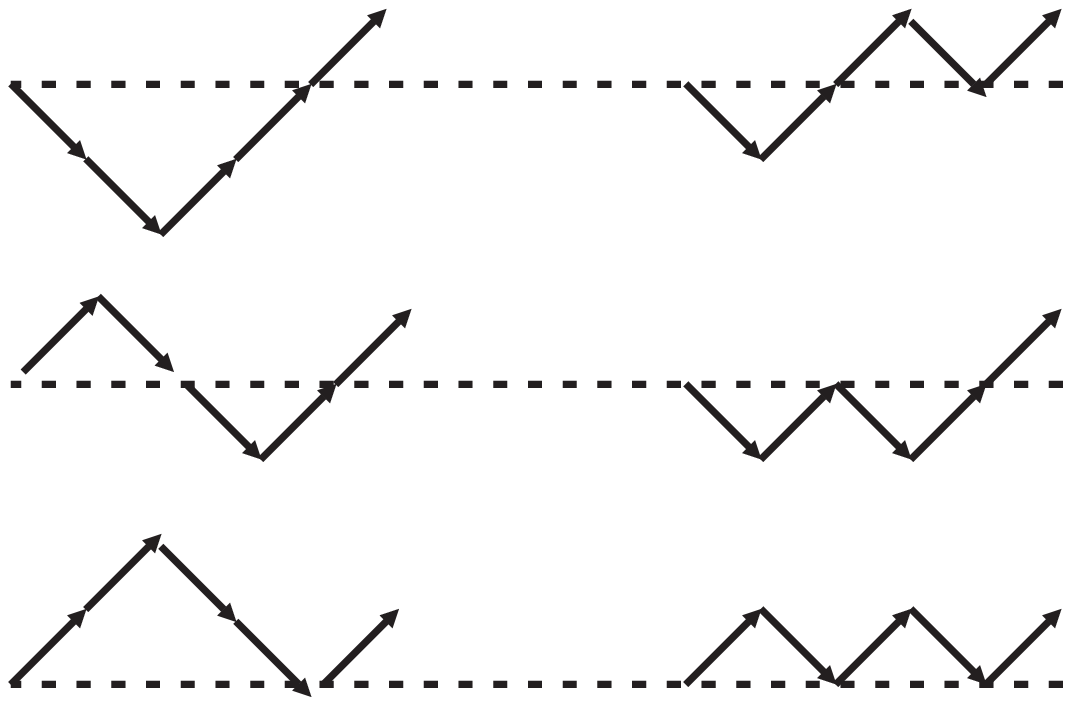}\\
Fig.5. An example of Corollary 4.5, where $n=2$
\end{center}

\section{Conclusions}
By Theorems 3.4 and 4.3, the lattice paths in Classes 1 and 2 have
the Chung-Feller properties of both Dyck type and Motzkin type. We
only prove the Chung-Feller theorem of Dyck type for Class 3. In
fact, the lattice paths in Class 3 have the Chung-Feller properties
of Motzkin type as well. We don't include this result in this paper
since the statements are very complicate.

 There are
many lattice paths which have the Chung-Feller properties of both
Dyck type and Motzkin type. For simplify, we set
$\mathcal{S}=\{(1,1),(1,-1)\}$, $w(s)=1$ for any $s\in \mathcal{S}$,
$l((1,1))=1$ and $l((1,-1))=0$. Let $\theta$ be a mapping from
$\mathscr{L}$ to $\mathbb{N}$, where $\mathscr{L}$ is the set of all
the $(\mathcal{S},1)$-lattice path. $\theta$ is called a parameter
on $(\mathcal{S},1)$-lattice path. For any $0\leq m\leq n$, if the
number of the $(\mathcal{S},1)$-lattice path $L$ with length $n$
such that $\theta(L)=m$ is independent on $m$, then we say that
$\theta$ is a {\it Chung-Feller parameter} for
$(\mathcal{S},1)$-lattice paths. There are two Chung-Feller
parameters on $(\mathcal{S},1)$-lattice path: non-positive pointed
length and absolute minimum pointed length. To end this paper, we
propose a problem: are there the other Chung-Feller parameters on
$(\mathcal{S},1)$-lattice paths?

 \end{document}